\documentclass[leqno,12pt]{amsart}
\usepackage{amsfonts}
\usepackage{amsmath,amssymb,amsthm}

\newcommand{\R}{\mathbb R}
\newcommand{\E}{\mathbb E}
\renewcommand{\span}{\mathrm{span}}
\newcommand{\tr}{\mathrm{tr}}
\everymath{\displaystyle}

\newtheorem{theorem}{Theorem}

\newtheorem{corollary}[theorem]{Corollary}

\begin{document}

\title[Meridian surfaces with Constant Mean Curvature in  $\E^4_2$]{Meridian Surfaces with Constant Mean Curvature in Pseudo-Euclidean
4-space with Neutral Metric}

\author{Bet\"{u}l Bulca and Velichka Milousheva}
\address{Uluda\u{g} University, Department of Mathematics, 16059 Bursa, Turkey}
\email{bbulca@uludag.edu.tr}
\address{Institute of Mathematics and Informatics, Bulgarian Academy of Sciences,
Acad. G. Bonchev Str. bl. 8, 1113, Sofia, Bulgaria;   "L. Karavelov"
Civil Engineering Higher School, 175 Suhodolska Str., 1373 Sofia,
Bulgaria} \email{vmil@math.bas.bg}


\begin{abstract}
In the present paper we consider a special class of Lorentz surfaces in the four-dimensional pseudo-Euclidean space with neutral metric which are 
one-parameter systems of meridians of rotational hypersurfaces with timelike or spacelike axis and call them meridian surfaces. 
We give the complete classification of  minimal and quasi-minimal meridian surfaces. We also classify the meridian surfaces with 
non-zero  constant mean curvature.
\end{abstract}

\maketitle

\section{Introduction}

The study of surfaces with constant mean curvature is one of the main topics in classical differential geometry which  goes back to the latter part of the 18th century. Lagrange was the first who found the minimal surface equation in 1761 when he looked  for a necessary condition for minimizing a certain integral. Actually, the notion of mean curvature was first formally defined by Meusnier in 1776. Throughout the 19th century grate mathematicians such as Gauss and Weierstrass devoted much of their studies to these surfaces. 
Constant mean curvature surfaces  in the 3-dimensional  Euclidean space are also  studied intensively nowadays by many geometers for their physical interpretation.
 For example, surfaces with constant curvature are important mathematical models of soap films and soap bubbles. 

Constant mean curvature surfaces (CMC surfaces) in arbitrary spacetime
are important objects for the special role they play in the theory of
general relativity. The study of CMC surfaces
involves not only geometric methods but also PDE
and complex analysis, that is why the theory of CMC surfaces is of
great interest not only for mathematicians but also for physicists
and engineers. Surfaces with constant mean curvature in
Minkowski space have been studied intensively in  the last years. See for example \cite{Liu-Liu-1}, \cite{Lop-2}, \cite{Sa}, \cite{Chav-Can}, \cite{Bran}.

In the four-dimensional pseudo-Euclidean space with neutral metric very few results are known on surfaces with constant mean curvature.  A special case of CMC surfaces are the quasi-minimal surfaces. A Lorentzian surface in a pseudo-Riemannian manifold is called quasi-minimal (pseudo-minimal or marginally trapped) if its mean curvature vector $H$ is lightlike at each point. Classification results on quasi-minimal surfaces in the pseudo-Euclidean space $\E^4_2$ have been obtained recently.
The classification of quasi-minimal surfaces with parallel mean
curvature vector in  $\E^4_2$ is given
in \cite{Chen-Garay}. In \cite{Chen1}  B.-Y. Chen classified
quasi-minimal Lorentz  flat surfaces in $\E^4_2$. As an application,
he gave the complete classification of biharmonic Lorentz surfaces
in  $\E^4_2$ with lightlike mean curvature vector. Several other
families of quasi-minimal surfaces  have also been classified. For
example, quasi-minimal surfaces with constant Gauss curvature in
$\E^4_2$ were classified in \cite{Chen2, Chen-Yang}. Quasi-minimal
Lagrangian surfaces and quasi-minimal slant surfaces in complex
space forms were classified, respectively, in \cite{Chen-Dillen} and
\cite{Chen-Mihai}. For an up-to-date survey on  quasi-minimal
surfaces, see also \cite{Chen3}.

In the present paper we construct special 2-dimensional Lorentz surfaces in $\E^4_2$ which are one-parameter systems of meridians of the rotational hypersurfaces with timelike or spacelike axis and call them meridian surfaces in $\E^4_2$. We describe all minimal meridian surfaces and show that all of them lie in hyperplanes of $\E^4_2$. We give the complete classification of  quasi-minimal meridian surfaces (Theorems \ref{T:quasi-minimal-a}, \ref{T:quasi-minimal-b}, and \ref{T:quasi-minimal-c}). We also classify the meridian surfaces with 
non-zero  constant mean curvature (Theorems \ref{T:CMC-a}, \ref{T:CMC-b}, and \ref{T:CMC-c}).

\section{Preliminaries}

Let $\mathbb{E}_{2}^{4}$ be the $4$-dimensional pseudo-Euclidean space with flat
metric of index $2$ given in local coordinates by
\begin{equation*}
\widetilde{g}=dx_{1}^{2}+dx_{2}^{2}-dx_{3}^{2}-dx_{4}^{2}, 
\end{equation*}
where $\left( x_{1},x_{2},x_{3},x_{4}\right) $ is a rectangular coordinate
system of $\mathbb{E}_{2}^{4}$.  We denote by $ \langle ., . \rangle$ the indefinite inner scalar product associated with  $\widetilde{g}$.
Since $\widetilde{g}$ is an indefinite
metric, a vector $v \in \E^4_2$ can have one of
the three casual characters: it can be \textit{spacelike} if $\langle v, v \rangle >0$
or $v=0$, \textit{timelike} if $\langle v, v \rangle<0$, and \textit{lightlike} if $\langle v, v \rangle =0$ and $v\neq 0$. This terminology is inspired by general relativity and the Minkowski 4-space $\E^4_1$.

We use the following denotations:
\begin{equation*}
\begin{array}{l}
\vspace{2mm}
\mathbb{S}^3_2(1) =\left\{V\in \E^4_2: \langle V, V \rangle =1 \right\}; \\
\vspace{2mm}
\mathbb{H}^3_1(-1) =\left\{ V\in \E^4_2: \langle V, V \rangle = -1\right\}.
\end{array}
\end{equation*}
The space $\mathbb{S}^3_2(1)$ is known as the de Sitter space, and the
space $\mathbb{H}^3_1(-1)$ is the hyperbolic space (or the anti-de Sitter space) \cite{O'N}.

Given a surface $M$ in $\mathbb{E}_{2}^{4}$, we denote by $g$
the induced metric of $\widetilde{g}$ on $M$. A surface $M$ in $\E^4_2$
 is called \emph{Lorentz}  if the
induced  metric $g$ on $M$ is Lorentzian, i.e. at each point $p\in M$ we have the following decomposition
$$\E^4_2 = T_pM \oplus N_pM$$
with the property that the restriction of the metric
onto the tangent space $T_pM$ is of
signature $(1,1)$, and the restriction of the metric onto the normal space $N_pM$ is of signature $(1,1)$.

We denote by $\nabla $ and $\overline{\nabla}$ the Levi-Civita connections
of $M$ and $\mathbb{E}_{2}^{4}$, respectively. Then for any vector fields $X,Y$
tangent to $M$ the Gauss formula is given by
\begin{equation*}
\overline{\nabla}_{X}Y=\nabla _{X}Y+h(X,Y),
\end{equation*}
where $h$ is the second fundamental form of $M$. If $D$ is the normal
connection on the normal bundle of $M$, then for any normal vector
field $\xi$ and any tangent vector field $X$ the Weingarten formula
is given by
\begin{equation*}
\overline{\nabla}_{X}\xi =-A_{\xi }X+D_{X}\xi,
\end{equation*}
where $A_{\xi}$ is the shape operator with respect to $\xi$. The shape operator and the second fundamental form are related by the
formula 
\begin{equation*}
\left \langle h(X,Y),\xi \right \rangle =\left \langle A_{\xi }X,Y\right
\rangle  
\end{equation*}
for any $X$ and $Y$ tangent to $M$ and any $\xi$ normal to $M$.

The mean curvature vector field $H$ of $M$ in $\mathbb{E}_{2}^{4}$
is defined as
\begin{equation*}
H=\frac{1}{2}\, \tr \, h.  \label{A8}
\end{equation*}
Thus, if $M$ is a
Lorentz surface and $\{X,Y\}$ is a local orthonormal frame  of
the tangent bundle such that $\langle X, X \rangle = 1$, $\langle Y, Y \rangle = -1$, then the mean curvature vector  field is given by
the formula $H = \frac{1}{2} \left(h(X,X) - h(Y,Y)\right)$.

A surface $M$ is called \textit{minimal} if its mean curvature vector vanishes
identically, i.e. $H=0$. A surface $M$ is called \textit{quasi-minimal}
(or \textit{pseudo-minimal}) if its mean curvature vector is lightlike at each point,
i.e. $H\neq 0$ and $\left \langle H,H\right \rangle =0$. Obviously,
quasi-minimal surfaces are always non-minimal. 

$M$ is said to have constant mean curvature  if $\langle H, H \rangle = const.$
We shall consider Lorentz surfaces in $\E^4_2$ for which  $\langle H, H \rangle = const \neq 0$. Such surfaces we call \textit{CMC surfaces}.

\section{Meridian Surfaces in Pseudo-Euclidean 4-Space}

Meridian surfaces in Euclidean 4-space we defined first in \cite{GM2}  as 2-dimensional surfaces lying on a
standard rotational hypersurface in $\mathbb{E}^{4}$. These surfaces are one-parameter systems of meridians of
the rotational hypersurface, that is why they are called \textit{meridian surfaces}.
The classification of  meridian surfaces with constant Gauss curvature, with constant mean curvature, Chen meridian surfaces and  meridian surfaces with parallel normal bundle is given in \cite{GM2}  and \cite{GM-BKMS}. Meridian surfaces in $\E^4$ having pointwise 1-type Gauss map are classified in
\cite{ABM}. 

The idea from the Euclidean case is used in \cite{GM6}  for the construction of special families of two-dimensional spacelike surfaces lying on rotational hypersurfaces with timelike or spacelike axis in the Minkowski space $\mathbb{E}_{1}^{4}$. These surfaces are called meridian surface of elliptic  or hyperbolic type, respectively. A local classification of marginally trapped meridian surfaces is given in \cite{GM6}.  Meridian surfaces in $\E^4_1$ with pointwise 1-type Gauss map are classified in \cite{AM}. The classification of  meridian surfaces of elliptic  or hyperbolic type with constant Gauss curvature, with constant mean curvature, Chen meridian surfaces and  meridian surfaces with parallel normal bundle is given in \cite{GM-MC}.

Following the idea from the Euclidean and Minkowski spaces, we shall construct Lorentz meridian surfaces in the pseudo-Euclidean 4-space $\E^4_2$  as one-parameter systems of meridians of rotational hypersurfaces with timelike or spacelike axis. 

\subsection{Meridian surfaces lying on a rotational hypersurface with timelike axis}

Let $Oe_1 e_2 e_3 e_4$ be the standard orthonormal frame  in $\E^4_2$, i.e. $\langle e_1,e_1 \rangle
=\langle e_2,e_2 \rangle =1, \langle e_3,e_3 \rangle = \langle e_4,e_4 \rangle = -1$. First we consider a
standard rotational hypersurface with timelike axis $Oe_4$. Similarly, we can consider a rotational hypersurface with axis $Oe_3$.

Since in the Minkowski space $\E^3_1 = \span\left \{ e_{1},e_{2},e_{3}\right \}$ there exist two types of  spheres, namely the  pseudo-sphere  $\mathbb{S}^2_1(1) =\left\{V\in \E^3_1: \langle V, V \rangle = 1\right \}$, i.e. the de Sitter space, and the pseudo-hyperbolic space $\mathbb{H}^2_1(-1)=\left\{V \in  \E^3_1: \langle V, V \rangle = - 1\right \}$, i.e. the anti-de Sitter space, we consider two types of rotational hypersurfaces about the axis  $Oe_4$.

\vskip 2mm
\textbf{Rotational hypersurface of first type.}

Let $f=f(u)$, $g=g(u)$ be smooth functions, defined in an interval $I\subset \R$. 
The first type rotational
hypersurface $\mathcal{M}^I$ in $\E^4_2$, obtained by the rotation of the
meridian curve $m:u\rightarrow (f(u),g(u))$ about the $Oe_4$-axis, is
parametrized as follows:
\begin{equation*}
\mathcal{M}^I:Z(u,w^1,w^2) = f(u)\left(\cosh w^1 \cos w^2 \,e_1 +   \cosh w^1 \sin w^2 \,e_2 + \sinh w^1 \,e_3\right) + g(u) \,e_4.
\end{equation*}
If we denote by $l^I(w^1,w^2) = \cosh w^1 \cos w^2 \,e_1 + \cosh w^1 \sin w^2 \,e_2 + \sinh w^1 \,e_3$ the unit position vector of the sphere
$\mathbb{S}^2_1(1)$ in $\E^3_1$ centered at the origin $O$, then the parametrization of the rotational hypersurface $\mathcal{M}^I$ is written as:
\begin{equation*}
\mathcal{M}^I:Z(u,w^1,w^2) = f(u) l^I(w^1,w^2) + g(u) \,e_4.
\end{equation*}
Now, we shall  construct Lorentz surfaces in $\E^4_2$ which are one-parameter systems of meridians of the hypersurface $\mathcal{M}^I$.

\vskip 2mm
First we consider a smooth spacelike curve $c:l=l(v)=l^I(w^1(v),w^2(v)), \,\, v \in J, \, J \subset \R$
on $\mathbb{S}^2_1(1)$ parametrized by the arc-length, i.e.  $\langle l',l'\rangle =1$.  We construct a two-dimensional  surface $\mathcal{M}'_a$  defined by:
\begin{equation}  \label{E:Eq-1-a}
\mathcal{M}'_a: z(u,v) = f(u)\, l(v) + g(u) \,e_4, \quad u \in I, \, v \in J.
\end{equation}
Since the surface $\mathcal{M}'_a$ is a
one-parameter system of meridians of $\mathcal{M}^I$,   we call
it a  \emph{meridian surface on $\mathcal{M}^I$}.

The tangent space of $\mathcal{M}'_a$ is spanned by the vector fields
$$z_u = f'(u) \,l(v) + g'(u) \,e_4; \qquad  z_v = f(u) \,l'(v),$$
so, the coefficients of the first fundamental form of $\mathcal{M}'$  are
$$E = \langle z_u, z_u \rangle =  f'^2 - g'^2; \quad F = \langle z_u, z_v \rangle = 0; \quad  G = \langle z_v, z_v \rangle = f^2.$$
Since we are interested in Lorentz surfaces, we assume that the meridian curve $m$ is timelike,  i.e. $f'^2 - g'^2 <0$. Without loss of generality we can assume that $f'^2 - g'^2 = -1$.
Then the coefficients of the first fundamental form are $E = -1; \, F = 0; \, G = f^2(u)$.
We consider the unit tangent vector fields  $X = z_u,\,\, Y = \frac{z_v}{f} = l'$,
which satisfy $\langle X,X \rangle = -1$,  $\langle Y,Y \rangle = 1$ and $\langle X,Y \rangle =0$.

Let  $t(v) = l'(v)$ be the tangent vector field of the curve $c$. Since $\langle t(v), t(v) \rangle = 1$,
 $\langle l(v), l(v) \rangle = 1$, and $\langle t(v), l(v) \rangle = 0$, then 
 there exists a unique (up to a sign)
 vector field $n(v)$, such that $\langle n(v), n(v) \rangle = -1$ and 
$\{l(v), t(v), n(v)\}$ is an orthonormal frame field in $\E^3_1 = \span\left \{ e_{1},e_{2},e_{3}\right \}$.  With respect to this
 frame field we have the following Frenet formulas of $c$ on $\mathbb{S}^2_1(1)$:
\begin{equation} \label{E:Eq-2}
\begin{array}{l}
\vspace{2mm}
l' = t;\\
\vspace{2mm}
t' = - \kappa \,n - l;\\
\vspace{2mm} n' = - \kappa \,t,
\end{array}
\end{equation}
 where $\kappa (v)= \langle t'(v), n(v) \rangle$ is the spherical curvature of $c$ on  $\mathbb{S}^2_1(1)$.
Now we consider the
following  normal vector fields:
\begin{equation} \label{E:Eq-2-2}
n_1 = n(v); \qquad n_2 = g'(u)\,l(v) + f'(u) \, e_4,
\end{equation}
which satisfy $\langle n_1, n_1 \rangle = -1$,
$\langle n_2, n_2 \rangle = 1$, $\langle n_1, n_2 \rangle =0$.

Taking into account \eqref{E:Eq-2} and \eqref{E:Eq-2-2} we get the following derivative formulas:
\begin{equation}  \label{E:Eq-3-0}
\begin{array}{ll}
\vspace{2mm}
\overline{\nabla}_X X = \kappa_m n_2; \qquad & \overline{\nabla}_X n_1 = 0; \\
\vspace{2mm}
\overline{\nabla}_X Y = 0; \quad & \overline{\nabla}_Y n_1 = -\frac{\kappa}{f}Y; \\
\vspace{2mm}
\overline{\nabla}_Y X = \frac{f'}{f}Y; \quad & \overline{\nabla}_X n_2 = \kappa_m X; \\
\vspace{2mm}
\overline{\nabla}_Y Y = \frac{f'}{f}X-\frac{\kappa}{f}n_1- \frac{g'}{f}n_2; \quad & \overline{\nabla}_Y n_2=\frac{g'}{f}Y,
\end{array}
\end{equation}
where $\kappa_m$ denotes the curvature of the meridian curve $m$,  which in  case of a timelike curve is given by the formula $\kappa_m (u)= f''(u) g'(u) - f'(u) g''(u)$.
Hence, we have 
\begin{equation} \label{E:Eq-3}
\begin{array}{l}
\vspace{2mm}
h(X,X) = \kappa_m\,n_2; \\
\vspace{2mm}
h(X,Y) =  0; \\
\vspace{2mm}
h(Y,Y) = - \frac{\kappa}{f}\,n_1 - \frac{g'}{f} \, n_2.
\end{array}
\end{equation}

Formulas \eqref{E:Eq-3} imply that the  mean curvature vector field $H$ of  $\mathcal{M}'_a$ is expressed as follows:
\begin{equation*}
H = -\frac{\kappa}{2f} \,n_1 - \frac{f \kappa_m +g'}{2f} \,n_2.
\end{equation*}
Using that $g'^2 = f'^2 + 1$ and $\kappa_m = \frac{f''}{g'}$, we obtain
\begin{equation} \label{E:Eq-Ha}
H=-\frac{\kappa}{2f} \,n_1 - \frac{f f''+(f')^2 + 1}{2f \sqrt{f'^2+1}} \,n_2.
\end{equation}

Now, let $c:l=l(v)=l^I(w^1(v),w^2(v)), \,\, v \in J, \, J \subset \R$ be a timelike curve
on $\mathbb{S}^2_1(1)$ parametrized by the arc-length, i.e.  $\langle l',l'\rangle = - 1$.  We consider the two-dimensional  surface $\mathcal{M}'_b$  defined by:
\begin{equation}  \label{E:Eq-1-b}
\mathcal{M}'_b: z(u,v) = f(u)\, l(v) + g(u) \,e_4, \quad u \in I, \, v \in J,
\end{equation}
where $f'^2-g'^2 = 1$.
The surface $\mathcal{M}'_b$ is another meridian surface on $\mathcal{M}^I$. 

In this case we consider an orthonormal frame field $\{l(v), t(v), n(v)\}$ of $\E^3_1$, such that $t = l'$,
$\langle l, l \rangle = 1$, $\langle t, t \rangle = -1$,
$\langle n, n \rangle = 1$. Now, the Frenet formulas of $c$ on $\mathbb{S}^2_1(1)$ are:
\begin{equation} \notag
\begin{array}{l}
\vspace{2mm}
l' = t;\\
\vspace{2mm}
t' =  \kappa \,n + l;\\
\vspace{2mm} n' =  \kappa \,t,
\end{array}
\end{equation}
 where $\kappa (v)= \langle t'(v), n(v) \rangle$.
The tangent vector fields of the meridian surface $\mathcal{M}'_b$  are
\begin{equation*}
z_u=f' \,l+g' \,e_4; \quad z_v = f\,t,
\end{equation*}
and since $f'^2-g'^2 = 1$, the coefficients of the first fundamental form are $E=1; \; F=0; \; G= - f^2(u)$.
We consider the orthonormal tangent frame $X=z_u$, $Y=\frac{z_v}{f}=t$ and the
orthonormal normal frame field defined by 
\begin{equation*}
n_1=n; \quad n_2=g' \,l+f'\, e_{4}.
\end{equation*}
Thus, we obtain a frame field $\left \{ X,Y,n_1,n_2\right \}$ of $\mathcal{M}'_b$ such that $\left \langle n_{1},n_{1}\right \rangle
=1,\left \langle n_{2},n_{2}\right \rangle =-1$ and $\left \langle
n_{1},n_{2}\right \rangle =0.$
With respect to this frame field we have the following derivative formulas:
\begin{equation} \label{E:Eq-3-b}
\begin{array}{ll}
\vspace{2mm}
\overline{\nabla}_X X = \kappa_m n_2; \qquad & \overline{\nabla}_X n_1 = 0; \\
\vspace{2mm}
\overline{\nabla}_X Y = 0; \quad & \overline{\nabla}_Y n_1 = \frac{\kappa}{f}Y; \\
\vspace{2mm}
\overline{\nabla}_Y X = \frac{f'}{f}Y; \quad & \overline{\nabla}_X n_2 = \kappa_m X; \\
\vspace{2mm}
\overline{\nabla}_Y Y = \frac{f'}{f}X + \frac{\kappa}{f}n_1- \frac{g'}{f}n_2; \quad & \overline{\nabla}_Y n_2=\frac{g'}{f}Y,
\end{array}
\end{equation}
where  $\kappa_m$ is the curvature of the meridian curve $m$, which in the case of a spacelike curve is given by the formula
$\kappa_m =  f' g'' - f'' g'$.

Equalities \eqref{E:Eq-3-b} imply that the mean curvature vector field is given  by the formula:
\begin{equation*}
H = -\frac{\kappa}{2f} \,n_1 + \frac{f \kappa_m +g'}{2f} \,n_2.
\end{equation*}
Having in mind that $g'^2 = f'^2 - 1$ and $\kappa_m = \frac{f''}{g'}$, we obtain
\begin{equation} \label{E:Eq-Hb}
H=-\frac{\kappa}{2f} \,n_1 + \frac{f f''+(f')^2 - 1}{2f \sqrt{f'^2-1}} \,n_2.
\end{equation}

So, we have two types of meridian surfaces lying on  $\mathcal{M}^I$:   meridian surfaces of type $\mathcal{M}'_a$ and  meridian surfaces of type $\mathcal{M}'_b$.

\vskip 2mm
\textbf{Rotational hypersurface of second type.}

Now, we consider the second type rotational
hypersurface $\mathcal{M}^{II}$ in $\E^4_2$, obtained by the rotation of the
meridian curve  $m:u\rightarrow (f(u),g(u))$  about the axis $Oe_4$, which is given by the following
parametrization:
\begin{equation*}
\mathcal{M}^{II}:Z(u,w^1,w^2) = f(u)\left(\sinh w^1 \cos w^2 \,e_1 +  \sinh w^1 \sin w^2 \,e_2 + \cosh w^1 \,e_3\right) + g(u) \,e_4.
\end{equation*}
Note that  $l^{II}(w^1,w^2) = \sinh w^1 \cos w^2 \,e_1 +  \sinh w^1 \sin w^2 \,e_2 + \cosh w^1 \,e_3$  is the unit position vector of the hyperbolic sphere
$\mathbb{H}^2_1(-1)$ in $\E^3_1 = \span\left \{ e_{1},e_{2},e_{3}\right \}$ centered at the origin $O$. So,  the parametrization of
$\mathcal{M}^{II}$ can be written as:
\begin{equation*}
\mathcal{M}^{II}:Z(u,w^1,w^2) = f(u) l^{II}(w^1,w^2) + g(u) \,e_4.
\end{equation*}

Meridian surfaces lying on the rotational hypersurface of second type $\mathcal{M}^{II}$ can be constructed as follows. Let $c: l = l(v) =  l^{II}(w^1(v),w^2(v))$ be a smooth curve on the hyperbolic sphere $\mathbb{H}^2_1(-1)$, where $w^1 = w^1(v)$, $w^2=w^2(v), \,\, v \in J, \, J \subset \R$.  Then the  two-dimensional surface $\mathcal{M}''$  defined by
\begin{equation}  \label{E:Eq-4}
\mathcal{M}'': z(u,v) = f(u)\, l(v) + g(u) \,e_4, \quad u \in I, \, v \in J
\end{equation}
 is a  one-parameter system of meridians of $\mathcal{M}^{II}$, which we call
a  \emph{meridian surface on $\mathcal{M}^{II}$}.

The tangent space of $\mathcal{M}''$ is spanned by the vector fields
$$z_u = f'(u) \,l(v) + g'(u) \,e_4; \quad  z_v = f(u) \,l'(v),$$
so the coefficients of the first fundamental form of $\mathcal{M}''$  are
$$E = \langle z_u, z_u \rangle =  -(f'^2 + g'^2); \quad F = \langle z_u, z_v \rangle = 0; \quad  G = \langle z_v, z_v \rangle = f^2 \langle l',l'\rangle.$$
Since $c$ is a curve lying on $\mathbb{H}^2_1(-1)$, we have $\langle l,l\rangle = -1$, which implies that the tangent vector field $t = l'$ satisfies $\langle t, t \rangle = 1$. Without loss of generality we suppose that $f'^2 + g'^2 =1$. Then the coefficients of the first fundamental form of $\mathcal{M}''$  are $E = -1; \, F =  0; \,  G = f^2$. 

We consider an  orthonormal frame field $\{l(v), t(v), n(v)\}$ of $c$ satisfying the conditions
$\langle l, l \rangle = -1$, $\langle t, t \rangle = 1$,
$\langle n, n \rangle = 1$. The Frenet formulas of $c$ on $\mathbb{H}^2_1(-1)$ are:
\begin{equation} \label{E:Eq-2-c}
\begin{array}{l}
\vspace{2mm}
l' = t;\\
\vspace{2mm}
t' =  \kappa \,n + l;\\
\vspace{2mm} n' = - \kappa \,t,
\end{array}
\end{equation}
 where $\kappa (v)= \langle t'(v), n(v) \rangle$ is the curvature of $c$ on  $\mathbb{H}^2_1(-1)$.

Let us consider the following orthonormal frame field of  $\mathcal{M}''$:
$$X = z_u; \quad Y = \frac{z_v}{f} = t; \quad n_1 = n(v); \quad n_2 = -g'(u)\,l(v) + f'(u) \, e_4.$$
This frame field satisfies $\langle X, X \rangle = -1$,
$\langle Y, Y \rangle = 1$, $\langle X, Y \rangle =0$, $\langle n_1, n_1 \rangle = 1$,
$\langle n_2, n_2 \rangle = -1$, $\langle n_1, n_2 \rangle =0$.
Taking into account  \eqref{E:Eq-2-c} we get the following derivative formulas:
\begin{equation} \notag
\begin{array}{ll}
\vspace{2mm}
\overline{\nabla}_X X = \kappa_m n_2; \qquad & \overline{\nabla}_X n_1 = 0; \\
\vspace{2mm}
\overline{\nabla}_X Y = 0; \quad & \overline{\nabla}_Y n_1 = -\frac{\kappa}{f}Y; \\
\vspace{2mm}
\overline{\nabla}_Y X = \frac{f'}{f}Y; \quad & \overline{\nabla}_X n_2 = -\kappa_m X; \\
\vspace{2mm}
\overline{\nabla}_Y Y = \frac{f'}{f}X + \frac{\kappa}{f}n_1 - \frac{g'}{f}n_2; \quad & \overline{\nabla}_Y n_2= -\frac{g'}{f}Y,
\end{array}
\end{equation}
where $\kappa_m = f' g'' -  f'' g'$. 
Hence, we obtain the formulas
\begin{equation*} 
\begin{array}{l}
\vspace{2mm}
h(X,X) = \kappa_m\,n_2; \\
\vspace{2mm}
h(X,Y) =  0; \\
\vspace{2mm}
h(Y,Y) =  \frac{\kappa}{f}\,n_1 - \frac{g'}{f} \, n_2,
\end{array}
\end{equation*}
which imply that the normal mean curvature vector field of $\mathcal{M}''$ is given by
\begin{equation} \notag
H = \frac{\kappa}{2f} \,n_1 + \frac{f f''+(f')^2 - 1}{2f \sqrt{1-f'^2}} \,n_2.
\end{equation}

Note that we can construct only one type of Lorentz meridian surfaces lying on the rotational hypersurface  $\mathcal{M}^{II}$.

\subsection{Meridian surfaces lying on a rotational hypersurface with spacelike axis}

In a similar way, we can construct meridian surfaces lying on a rotational
hypersurface  with spacelike axis  $Oe_1$ (or $Oe_2$). 

In the Minkowski space $\E^3_2 = \span\left \{ e_{2},e_{3},e_{4}\right \}$  there exist two types of spheres, namely the  de Sitter space $\mathbb{S}^2_2(1) =\left\{V\in \E^3_2: \langle V, V \rangle = 1\right \}$,  and the hyperbolic space $\mathbb{H}^2_1(-1)=\left\{V \in  \E^3_2: \langle V, V \rangle = - 1\right \}$.  So, we can consider two types of rotational hypersurfaces about the axis  $Oe_1$.

\vskip 2mm
\textbf{Rotational hypersurface of first type.}

Let  $f=f(u)$, $g=g(u)$ be smooth functions, defined in an interval $I\subset \R$. We denote by 
$\tilde{l}^I(w^1,w^2) = \cosh w^1 \,e_2 +   \sinh w^1 \cos w^2 \,e_3 + \sinh w^1 \sin w^2 \,e_4$  the unit position vector of the sphere
$\mathbb{S}^2_2(1)$ in $\E^3_2 = \span\left \{ e_{2},e_{3},e_{4}\right \}$ centered at the origin $O$.
Then, the first type rotational
hypersurface $\widetilde{\mathcal{M}}^I$, obtained by the rotation of the
meridian curve $m:u\rightarrow (f(u),g(u))$ about the axis $Oe_1$, is
parametrized by
\begin{equation*}
\widetilde{\mathcal{M}}^I:Z(u,w^1,w^2) = g(u) \,e_1 + f(u)\left(\cosh w^1 \,e_2 +   \sinh w^1 \cos w^2 \,e_3 + \sinh w^1 \sin w^2 \,e_4\right),
\end{equation*}
or equivalently,
\begin{equation*}
\widetilde{\mathcal{M}}^I:Z(u,w^1,w^2) =  g(u) \,e_1 + f(u) \,\tilde{l}^I(w^1,w^2).
\end{equation*}

Lorentz meridian surfaces lying on  $\widetilde{\mathcal{M}}^I$ are  one-parameter systems of meridians of $\widetilde{\mathcal{M}}^I$. They can be constructed as follows.
Let  $c: l = l(v) =  \tilde{l}^I(w^1(v),w^2(v)), \, v \in J, \, J \subset \R$ be a smooth curve on  $\mathbb{S}^2_2(1)$.
We consider the two-dimensional
 surface $\widetilde{\mathcal{M}}'$  defined by:
\begin{equation}  \notag
\widetilde{\mathcal{M}}': z(u,v) = g(u) \,e_1 + f(u)\, l(v), \quad u \in I, \, v \in J.
\end{equation}
It is a one-parameter system of meridians of $\widetilde{\mathcal{M}}^I$, so we call $\widetilde{\mathcal{M}}'$  a  \emph{meridian surface on $\widetilde{\mathcal{M}}^I$}.

It can easily be seen that the meridian surface 
 $\mathcal{M}''$, defined by \eqref{E:Eq-4}, can be transformed into the surface $\widetilde{\mathcal{M}}'$ by the transformation $T$ given by the matrix
\begin{equation} \label{E:Eq-T}
T = \left(\begin{array}{cccc}
\vspace{2mm}
0 & 0& 0&1\\
\vspace{2mm}
0 & 0& 1&0\\
\vspace{2mm}
1 & 0& 0&0\\
\vspace{2mm}
 0& 1& 0&0\\
\end{array}\right).
\end{equation}
So, the meridian surface  $\mathcal{M}''$ lying on $\mathcal{M}^{II}$  and the meridian surface $\widetilde{\mathcal{M}}'$ lying on $\widetilde{\mathcal{M}}^I$ are congruent. Hence, all results concerning the surface  $\mathcal{M}''$  hold true for the surface $\widetilde{\mathcal{M}}'$.

\vskip 2mm
\textbf{Rotational hypersurface of second type.}

The second type rotational hypersurface $\widetilde{\mathcal{M}}^{II}$, obtained by the rotation of the
meridian curve $m$ about $Oe_1$, is given by the following parametrization:
\begin{equation*}
\widetilde{\mathcal{M}}^{II}:Z(u,w^1,w^2) = g(u) \,e_1 + f(u)\left(\sinh w^1 \,e_2 + \cosh w^1 \cos w^2 \,e_3 + \cosh w^1 \sin w^2 \,e_4\right).
\end{equation*}
Here, $\tilde{l}^{II}(w^1,w^2) = \sinh w^1 \,e_2 + \cosh w^1 \cos w^2 \,e_3 + \cosh w^1 \sin w^2 \,e_4$ is the unit position vector of the hyperbolic sphere
$\mathbb{H}^2_1(-1)$ in $\E^3_2 = \span\left \{ e_{2},e_{3},e_{4}\right \}$ centered at the origin $O$. So, the parametrization of $\widetilde{\mathcal{M}}^{II}$ can be written as
\begin{equation*}
\widetilde{\mathcal{M}}^{II}:Z(u,w^1,w^2) =  g(u) \,e_1 + f(u) \,\tilde{l}^{II}(w^1,w^2).
\end{equation*}

We can construct two types of meridian surfaces lying on the second type rotational hypersurface  $\widetilde{\mathcal{M}}^{II}$. 

First, we consider a smooth spacelike curve $c: l = l(v) =  \tilde{l}^{II}(w^1(v),w^2(v)), \, v \in J, \, J \subset \R$ lying on the hyperbolic sphere $\mathbb{H}^2_1(-1)$ in $\E^3_2$. We assume that $c$ is parametrized by the arc-length, i.e.   $\langle l',l'\rangle =1$. Let $\widetilde{\mathcal{M}}''_a$ be the surface lying on $\widetilde{\mathcal{M}}^{II}$ and defined by:
\begin{equation}  \notag
\widetilde{\mathcal{M}}''_a: z(u,v) = g(u) \,e_1 + f(u)\, l(v), \quad u \in I, \, v \in J.
\end{equation}
The tangent space of $\widetilde{\mathcal{M}}''_a$  is spanned by the vector fields
$$z_u = g'(u) \,e_1 + f'(u) \,l(v); \qquad  z_v = f(u) \,l'(v),$$
so, the coefficients of the first fundamental form  are
$$E = g'^2  -  f'^2; \quad F = 0; \quad  G = f^2.$$
Since we are interested in Lorentz surfaces,  we assume that  $f'^2 - g'^2 = 1$. 
Then the coefficients of the first fundamental form are $E = -1; \, F = 0; \, G = f^2$.
It is easy to see that under the tranformation $T$ given by \eqref{E:Eq-T} the surface $\mathcal{M}'_b$ is transformed into the surface $\widetilde{\mathcal{M}}''_a$. So, all results concerning the surface   $\mathcal{M}'_b$  hold true for the surface $\widetilde{\mathcal{M}}''_a$.

Second, we consider a  timelike curve $c: l = l(v) =  \tilde{l}^{II}(w^1(v),w^2(v)), \, v \in J, \, J \subset \R$ lying on the hyperbolic sphere $\mathbb{H}^2_1(-1)$ and parametrized by the arc-length, i.e.  $\langle l',l'\rangle = -1$. 
Then the surface $\widetilde{\mathcal{M}}''_b$ defined by:
\begin{equation}  \notag
\widetilde{\mathcal{M}}''_b: z(u,v) = g(u) \,e_1 + f(u)\, l(v), \quad u \in I, \, v \in J,
\end{equation}
where $f'^2 - g'^2 = -1$ is a Lorentz meridian surface lying on $\widetilde{\mathcal{M}}^{II}$.
It is clear that the meridian surfaces $\mathcal{M}'_a$ and 
$\widetilde{\mathcal{M}}''_b$ are congruent up to the transformation $T$.

\vskip 2mm
So, it is worth studying three types of Lorentz meridian surfaces in $\E^4_2$, namely the surfaces denoted by $\mathcal{M}'_a$, $\mathcal{M}'_b$, and $\mathcal{M}''$.

\section{Minimal Meridian Surfaces  in $\E^4_2$}

In this section we give the classification of all minimal meridian surfaces in $\E^4_2$.

\begin{theorem} \label{T:minimal-a}
Let  $\mathcal{M}'_a$ be a meridian surface on $\mathcal{M}^I$ defined by \eqref{E:Eq-1-a}. Then
$\mathcal{M}'_a$ is minimal if and only if the curve $c$   has zero spherical curvature and the meridian curve $m$ is given by 
$$f(u) = \pm \sqrt{-u^2 +2au+b}, \quad g(u) = \pm \sqrt{a^2+b} \arcsin \frac{u-a}{\sqrt{a^2+b}} + c,$$ 
where $a = const$, $b=const$, $c = const$. 
\end{theorem}

\begin{proof}
The mean curvature vector field $H$ of the meridian surface $\mathcal{M}'_a$ is given by formula  \eqref{E:Eq-Ha}. Hence, $\mathcal{M}'_a$ is minimal if and only if the curvature of $c$ is $\kappa = 0$ and the function $f(u)$ satisfies the following equation 
$$f f''+(f')^2 + 1 =0.$$
The solutions of this differential equation are given by the formula 
$f(u) = \pm \sqrt{-u^2 +2au+b}$, where $a=const$, $b =const$. Having in mind that $g' = \sqrt{f'^2 + 1}$, we get the following equation for $g(u)$:
$$g' = \pm \frac{\sqrt{a^2+b}}{\sqrt{-u^2 +2au+b}}.$$
Integrating the above equation we obtain
$$g(u) = \pm \sqrt{a^2+b} \arcsin \frac{u-a}{\sqrt{a^2+b}} + c, \quad c = const.$$

\end{proof}

Note that if $\mathcal{M}'_a$ is minimal, then $\kappa=0$ and  from \eqref{E:Eq-3-0} we get that $\overline{\nabla}_X n_1 = 0, \; \overline{\nabla}_Y n_1 = 0$. This  means that the normal vector field $n_1$ is constant. Hence,  the surface $\mathcal{M}'_a$  lies in the constant 3-dimensional space $\E^3_1 = \span \{X,Y,n_2\}$. Consequently, $\mathcal{M}'_a$ lies in a hyperplane of $\E^4_2$.

\begin{theorem} \label{T:minimal-b}
Let  $\mathcal{M}'_b$ be a meridian surface on $\mathcal{M}^I$ defined by \eqref{E:Eq-1-b}. Then
$\mathcal{M}'_b$ is minimal if and only if the curve $c$   has zero spherical curvature and the meridian curve $m$ is given by 
$$f(u) = \pm \sqrt{u^2 +2au+b}, \quad g(u) = \pm \sqrt{a^2-b} \ln |u+a + \sqrt{u^2 +2au+b}| + c,$$ 
where $a = const$, $b=const$, $c = const$, $a^2-b>0$.
\end{theorem}

\begin{proof}
Using that the mean curvature vector field $H$ of $\mathcal{M}'_b$ is given by formula  \eqref{E:Eq-Hb}, we get that $\mathcal{M}'_b$ is minimal if and only if $\kappa = 0$ and the function $f(u)$ satisfies the equation 
$$f f''+(f')^2 -1 =0.$$
The solutions of this differential equation are given by the formula 
$f(u) = \pm \sqrt{u^2 +2au+b}$, where $a=const$, $b =const$. Using that $g' = \sqrt{f'^2 - 1}$, we get the following equation for $g(u)$:
$$g' = \pm \frac{\sqrt{a^2-b}}{\sqrt{u^2 +2au+b}}.$$
Integrating the last equation we obtain
$$g(u) = \pm \sqrt{a^2-b} \ln |u+a + \sqrt{u^2 +2au+b}| + c, \quad c = const.$$

\end{proof}

Note that if $\mathcal{M}'_b$ is minimal, we have $\kappa=0$ and again we obtain  $\overline{\nabla}_X n_1 = 0, \;\overline{\nabla}_Y n_1 = 0$, i.e.  the normal vector field $n_1$ is constant. In this case the surface $\mathcal{M}'_b$  lies in the constant 3-dimensional space $\E^3_2 = \span \{X,Y,n_2\}$. Consequently, $\mathcal{M}'_b$ lies in a hyperplane of $\E^4_2$.

\begin{theorem} \label{T:minimal-c}
Let  $\mathcal{M}''$ be a meridian surface on $\mathcal{M}^{II}$ defined by \eqref{E:Eq-4}. Then
  $\mathcal{M}''$ is minimal if and only if the curve $c$   has zero spherical curvature and the meridian curve $m$ is given by 
$$f(u) = \pm \sqrt{u^2 +2au+b}, \quad g(u) = \pm \sqrt{b -a^2} \ln |u+a + \sqrt{u^2 +2au+b}| + c,$$ 
where $a = const$, $b=const$, $c = const$, $b -a^2>0$.
\end{theorem}

The proof is similar to the proof of the previous two theorems. Again we have that if  $\mathcal{M}''$ is minimal, then it lies in a constant 3-dimensional hyperplane $\E^3_2$ of $\E^4_2$.

\vskip 2mm
Finally, we can formulate the following

\begin{corollary}
There are no minimal meridian surfaces lying fully in $\E^4_2$.

\end{corollary}

\section{Quasi-Minimal Meridian Surfaces  in $\E^4_2$}

In this section we classify all  quasi-minimal meridian surfaces in the pseudo-Euclidean 4-space $\E^4
_2$.

\begin{theorem} \label{T:quasi-minimal-a}
Let  $\mathcal{M}'_a$ be a meridian surface on $\mathcal{M}^I$ defined by \eqref{E:Eq-1-a}. Then
$\mathcal{M}'_a$ is quasi-minimal if and only if the curve $c$   has  constant curvature $\kappa = a = const,  \, a \neq 0$ and the meridian curve $m$ is determined by $f' = \varphi(f)$ where
\begin{equation} \notag
\varphi(t) = \pm \frac{1}{t} \sqrt{(\pm a t+ c)^2 -t^2},  \quad c = const,
\end{equation}
$g(u)$ is defined by $g' = \sqrt{f'^2+1}$. 
\end{theorem}

\begin{proof}
Using formula  \eqref{E:Eq-Ha} for the mean curvature vector field $H$ of the meridian surface $\mathcal{M}'_a$, we get that $\mathcal{M}'_a$ is quasi-minimal, i.e. $H \neq 0$ and $\langle H, H \rangle =0$, if and only if 
$$\frac{\left(f f''+(f')^2 + 1\right)^2 }{f'^2 +1}= \kappa^2, \quad \kappa \neq 0.$$
The left-hand side of this equation is a function of $u$, the right-hand side of the equation is a function of $v$. Hence, we obtain
\begin{equation*} 
\begin{array}{l}
\vspace{2mm}
\kappa = a, \quad a = const \neq 0; \\
\vspace{2mm}
\left(f f''+(f')^2 + 1\right)^2 = a^2 (f'^2 +1).
\end{array}
\end{equation*}
So, the meridian $m$ is determined by the following differential
equation:
\begin{equation} \label{E:Eq-5}
f f''+(f')^2 + 1= \pm  a \sqrt{f'^2 + 1}.
\end{equation}
The solutions of the above  differential equation  can be found in the  following way.
If we set $f' = \varphi (f)$ in equation \eqref{E:Eq-5},  we  obtain
that the function $\varphi  = \varphi (t)$ is a solution of the equation:
\begin{equation} \label{E:Eq-6}
\frac{t}{2} \,(\varphi ^2)' + \varphi ^2 + 1 = \pm a  \sqrt{\varphi ^2 + 1}.
\end{equation}
Now, we  set $z(t) = \sqrt{\varphi ^2(t) +1}$, so, equation \eqref{E:Eq-6} takes the form
\begin{equation} \notag
z' + \frac{1}{t}\, z = \pm \frac{a}{t}.
\end{equation}
The general solution of the last equation is given by the formula $z(t) = \frac{\pm at +c}{t}$, $ c = const$.
Hence, the general solution of \eqref{E:Eq-6} is
\begin{equation} \notag
\varphi(t) = \pm \frac{1}{t} \sqrt{(\pm a t+ c)^2 -t^2}.
\end{equation}

\end{proof}

Similarly to the proof of Theorem \ref{T:quasi-minimal-a} we obtain the classification of all quasi-minimal meridian surface of  type $\mathcal{M}'_b$.

\begin{theorem} \label{T:quasi-minimal-b}
Let  $\mathcal{M}'_b$ be a meridian surface on $\mathcal{M}^I$ defined by \eqref{E:Eq-1-b}. Then
$\mathcal{M}'_b$ is quasi-minimal if and only if the curve $c$   has constant curvature $\kappa = a = const,  \, a \neq 0$ and the meridian curve $m$ is determined by $f' = \varphi(f)$ where
\begin{equation} \notag
\varphi(t) = \pm \frac{1}{t} \sqrt{(c \pm a t)^2 + t^2},   \quad c = const,
\end{equation}
$g(u)$ is defined by $g' = \sqrt{f'^2-1}$. 
\end{theorem}

The classification of all quasi-minimal meridian surface of  type $\mathcal{M}''$  is given in the next theorem.

\begin{theorem} \label{T:quasi-minimal-c}
Let  $\mathcal{M}''$ be a meridian surface on $\mathcal{M}^{II}$ defined by \eqref{E:Eq-4}. Then
$\mathcal{M}''$ is quasi-minimal if and only if the curve $c$   has  constant curvature $\kappa = a = const, \, a \neq 0$ and the meridian curve $m$ is determined by $f' = \varphi(f)$ where
\begin{equation} \notag
\varphi(t) = \pm \frac{1}{t} \sqrt{t^2 - (c \pm a t)^2},   \quad c = const,
\end{equation}
$g(u)$ is defined by $g' = \sqrt{1-f'^2}$. 
\end{theorem}

The proof is similar to the proof of Theorem \ref{T:quasi-minimal-a}.

\section{Meridian Surfaces with Constant Mean Curvature in $\mathbb{E}_{2}^{4}$}

In this section we shall classify all meridian surfaces with non-zero constant mean curvature.

Let  $\mathcal{M}'_a$ be a meridian surface on $\mathcal{M}^I$ defined by \eqref{E:Eq-1-a}. The mean curvature vector field $H$ of 
$\mathcal{M}'_a$ is given by formula  \eqref{E:Eq-Ha}. So, we have
\begin{equation} \label{E:Eq-7}
\langle H, H \rangle = \frac{\left(f f''+(f')^2 + 1\right)^2 -\kappa^2 (f'^2 +1)}{4f^2(f'^2 +1)}.
\end{equation}

In the next theorem we classify the meridian surfaces of type  $\mathcal{M}'_a$ for which $\langle H, H \rangle = c= const, \; c \neq 0$.

\begin{theorem} \label{T:CMC-a}
Let $\mathcal{M}'_a$ be a meridian surface on $\mathcal{M}^I$ defined by \eqref{E:Eq-1-a}.
Then  $\mathcal{M}'_a$  has non-zero constant  mean curvature, i.e.  $\langle H, H \rangle = c =const, \, c\neq 0$, if and only if the curve $c$   has constant curvature $\kappa = a = const, \; a \neq 0$ and the meridian curve $m$ is determined by  
$f' = \varphi(f)$ where
\begin{equation} \notag
\begin{array}{l}
\vspace{2mm}
\varphi(t) = \pm \frac{1}{t} \sqrt{ \left(b \pm \frac{t}{2} \sqrt{a^2 +4ct^2} \pm \frac{a^2}{4 \sqrt{c}} \ln | 2\sqrt{c} t +  \sqrt{a^2 +4ct^2}| \right)^2 - t^2  }, \quad  {\rm if} \; c >0, \\
\vspace{2mm}
\varphi(t) = \pm \frac{1}{t} \sqrt{ \left(b \pm \frac{t}{2} \sqrt{a^2 - 4ct^2} \pm \frac{a^2}{4 \sqrt{-c}} \arcsin \frac{2\sqrt{-c}\,t}{a} \right)^2 - t^2  }, \quad  {\rm if} \; c <0,
\end{array}
\end{equation}
$b = const$, and $g(u)$ is defined by $g' = \sqrt{f'^2+1}$. 

\end{theorem}

\begin{proof}
Let $\mathcal{M}'_a$ be a meridian surface on $\mathcal{M}^I$. It follows from \eqref{E:Eq-7} that the condition on $\mathcal{M}'_a$ to have non-zero constant mean curvature, i.e. $\langle H, H \rangle = c =const, \, c\neq 0$ is equivalent to the equality 
\begin{equation} \notag
\frac{\left(f f''+(f')^2 + 1\right)^2 -\kappa^2 (f'^2 +1)}{4f^2(f'^2 +1)} = c.
\end{equation}
The last equality can be written as 
\begin{equation} \label{E:Eq-8}
\frac{\left(f f''+(f')^2 + 1\right)^2 - 4 c f^2 (f'^2 +1)}{f'^2 +1} = \kappa^2.
\end{equation}
Since the left-hand side of \eqref{E:Eq-8} is a function of $u$, the right-hand side of \eqref{E:Eq-8}  is a function of $v$,  we obtain
\begin{equation*} 
\begin{array}{l}
\vspace{2mm}
\kappa = a, \quad a = const \neq 0; \\
\vspace{2mm}
\frac{\left(f f''+(f')^2 + 1\right)^2 - 4 c f^2 (f'^2 +1)}{f'^2 +1}= a^2.
\end{array}
\end{equation*}
Hence, the meridian $m$ is determined by the solutions of the following differential
equation:
\begin{equation*} 
\left(f f''+(f')^2 + 1\right)^2 - 4 c f^2 (f'^2 +1)= a^2 (f'^2 +1).
\end{equation*}
We can find the solutions of the last equation  by setting
$f' = \varphi (f)$.  Then  we  obtain
that the function $\varphi  = \varphi (t)$ is a solution of the equation:
\begin{equation} \label{E:Eq-9}
\left(\frac{t}{2} \,(\varphi ^2)' + \varphi ^2 + 1 \right)^2 - 4 c t^2 (\varphi^2 +1)= a^2 (\varphi ^2 + 1).
\end{equation}
If we set $z(t) = \sqrt{\varphi ^2(t) +1}$,  from equation \eqref{E:Eq-9} we get
\begin{equation} \notag
z' + \frac{1}{t}\, z = \pm \frac{\sqrt{a^2 +4ct^2}}{t}.
\end{equation}
The general solution of the last equation is given by the formula $z(t) = \frac{1}{t} \left(b \pm \int{\sqrt{a^2 +4ct^2}} \,dt \right)$, where $b= const$.

If $c>0$, i.e. the mean curvature vector field $H$ is spacelike, then 
$\int{\sqrt{a^2 +4ct^2}} \,dt =  \frac{t}{2} \sqrt{a^2 +4ct^2} + \frac{a^2}{4 \sqrt{c}} \ln | 2\sqrt{c} t +  \sqrt{a^2 +4ct^2}|$.
Hence, $$z(t) = \frac{1}{t}  \left(b \pm  \frac{t}{2} \sqrt{a^2 +4ct^2} \pm \frac{a^2}{4 \sqrt{c}} \ln | 2\sqrt{c} t +  \sqrt{a^2 +4ct^2}| \right)$$
and  the general solution of \eqref{E:Eq-9} is
\begin{equation} \notag
\varphi(t) = \pm \frac{1}{t} \sqrt{ \left(b \pm \frac{t}{2} \sqrt{a^2 +4ct^2} \pm \frac{a^2}{4 \sqrt{c}} \ln | 2\sqrt{c} t +  \sqrt{a^2 +4ct^2}| \right)^2 - t^2}.
\end{equation}

If $c<0$, i.e. the mean curvature vector field $H$ is timelike, then 
$\int{\sqrt{a^2 +4ct^2}} \,dt =  \frac{t}{2} \sqrt{a^2 - 4ct^2} + \frac{a^2}{4 \sqrt{-c}} \arcsin \frac{2\sqrt{-c}\,t}{a}$.
Hence, the function $z(t)$ is given by 
$$z(t) = \frac{1}{t}  \left(b \pm  \frac{t}{2} \sqrt{a^2 - 4ct^2} \pm \frac{a^2}{4 \sqrt{-c}} \arcsin \frac{2\sqrt{-c}\,t}{a} \right)$$
and  the general solution of \eqref{E:Eq-9} is
\begin{equation} \notag
\varphi(t) = \pm \frac{1}{t} \sqrt{ \left(b \pm \frac{t}{2} \sqrt{a^2 - 4ct^2} \pm \frac{a^2}{4 \sqrt{-c}} \arcsin \frac{2\sqrt{-c}\,t}{a} \right)^2 - t^2  }.
\end{equation}

\end{proof}

Now, let  $\mathcal{M}'_b$ be a meridian surface on $\mathcal{M}^I$ defined by \eqref{E:Eq-1-b}. Hence, the mean curvature vector field $H$ of 
$\mathcal{M}'_b$ is given by formula  \eqref{E:Eq-Hb} and  
\begin{equation*} \label{E:Eq-7-b}
\langle H, H \rangle = \frac{\kappa^2 (f'^2 -1)-\left(f f''+(f')^2 - 1\right)^2}{4f^2(f'^2 -1)}.
\end{equation*}

The classification of the meridian surfaces of type  $\mathcal{M}'_b$ for which $\langle H, H \rangle = c= const, \; c \neq 0$ is given in the following theorem.

\begin{theorem} \label{T:CMC-b}
Let $\mathcal{M}'_b$ be a meridian surface on $\mathcal{M}^I$ defined by \eqref{E:Eq-1-b}.
Then  $\mathcal{M}'_b$  has non-zero constant  mean curvature, i.e.  $\langle H, H \rangle = c =const, \, c\neq 0$, if and only if the curve $c$   has constant curvature $\kappa = a = const, \; a \neq 0$ and the meridian curve $m$ is determined by  
$f' = \varphi(f)$ where
\begin{equation} \notag
\begin{array}{l}
\vspace{2mm}
\varphi(t) = \pm \frac{1}{t} \sqrt{ \left(b \pm \frac{t}{2} \sqrt{a^2 - 4ct^2} \pm \frac{a^2}{4 \sqrt{c}} \arcsin \frac{2\sqrt{c}\,t}{a} \right)^2 + t^2  }, \quad  {\rm if} \; c >0, \\
\vspace{2mm}
\varphi(t) = \pm \frac{1}{t} \sqrt{ \left(b \pm \frac{t}{2} \sqrt{a^2 -4ct^2} \pm \frac{a^2}{4 \sqrt{-c}} \ln | 2\sqrt{-c} t +  \sqrt{a^2 -4ct^2}| \right)^2 + t^2  }, \quad  {\rm if} \; c <0,
\end{array}
\end{equation}
$b = const$, and $g(u)$ is defined by $g' = \sqrt{f'^2-1}$. 
\end{theorem}

The proof of this theorem is similar to the proof of Theorem \ref{T:CMC-a}.

At the end of this section we give the classification of  the meridian surfaces of type  $\mathcal{M}''$ for which $\langle H, H \rangle = c= const, \; c \neq 0$.

\begin{theorem} \label{T:CMC-c}
Let  $\mathcal{M}''$ be a meridian surface on $\mathcal{M}^{II}$ defined by \eqref{E:Eq-4}. Then  $\mathcal{M}''$ 
  has non-zero constant  mean curvature, i.e.  $\langle H, H \rangle = c =const, \, c\neq 0$, if and only if the curve $c$   has constant curvature $\kappa = a = const, \; a \neq 0$ and the meridian curve $m$ is determined by  
$f' = \varphi(f)$ where
\begin{equation} \notag
\begin{array}{l}
\vspace{2mm}
\varphi(t) = \frac{1}{t} \sqrt{t^2 - \left(b \mp \frac{t}{2} \sqrt{a^2 - 4ct^2} \mp \frac{a^2}{4 \sqrt{c}} \arcsin \frac{2\sqrt{c}\,t}{a} \right)^2}, \quad  {\rm if} \; c >0, \\
\vspace{2mm}
\varphi(t) = \pm \frac{1}{t} \sqrt{t^2- \left(b \mp \frac{t}{2} \sqrt{a^2 -4ct^2} \mp \frac{a^2}{4 \sqrt{-c}} \ln | 2\sqrt{-c} t +  \sqrt{a^2 -4ct^2}| \right)^2 }, \quad  {\rm if} \; c <0,
\end{array}
\end{equation}
$b = const$, and $g(u)$ is defined by $g' = \sqrt{1-f'^2}$. 
\end{theorem}

 \vskip 5mm \textbf{Acknowledgments:}
The second author is partially supported by the Bulgarian National Science Fund,
Ministry of Education and Science of Bulgaria under contract
DFNI-I 02/14.

The paper is prepared
during the first author's visit at the Institute of
Mathematics and Informatics at the Bulgarian Academy of Sciences,
Sofia, Bulgaria  in November -- December 2015.


\bigskip



\begin{thebibliography}{99}


\bibitem{ABM}
Arslan K., Bulca B., Milousheva V., \textit{Meridian surfaces in
$\E^4$ with pointwise 1-type Gauss map},  Bull. Korean Math. Soc., \textbf{51}, no. 3 (2014),  911--922.

\bibitem{AM}
Arslan K., Milousheva V., \textit{Meridian surfaces of elliptic or hyperbolic type with pointwise 1-type Gauss map in Minkowski 4-space}, Taiwanese J. Math., \textbf{20}, no. 2 (2016), 311--332.

\bibitem{Bran}
Brander D., \emph{Singularities of spacelike constant mean curvature surfaces in Lorentz-Minkowski space}.
Math. Proc. Camb. Phil. Soc.  \textbf{150} (2011),  527--556.

\bibitem{Chav-Can}
Chaves R., C\^{a}ndido, C., \emph{The Gauss map of spacelike rotational surfaces with constant mean curvature in the Lorentz-Minkowski space}.
Differential geometry, Valencia, 2001, 106--114, World Sci. Publ., River Edge, NJ, 2002.


\bibitem{Chen1}
Chen B.-Y., \textit{Classification of marginally trapped Lorentzian flat
surfaces in $\E^4_2$ and its application to biharmonic surfaces},
J. Math. Anal.  Appl., 2008,  340(2), 861--875.


\bibitem{Chen2}
Chen B.-Y., \textit{Classification of marginally trapped surfaces of
constant curvature in Lorentzian complex plane}, Hokkaido
Math. J., 2009, 38(2), 361--408.

\bibitem{Chen3}
Chen B.-Y., \textit{Black holes, marginally trapped surfaces and
quasi-minimal surfaces}, Tamkang J. Math., 2009, 40(4), 313--341.

\bibitem{Chen-Dillen}
Chen B.-Y., Dillen F., \textit{Classification of marginally trapped
Lagrangian surfaces in Lorentzian complex space forms}, J. Math. Phys., 2007, 48(1), 013509, 23 pp.; Erratum,
J. Math. Phys., 2008, 49(5), 059901, 1p.

\bibitem{Chen-Garay}
Chen B.-Y., Garay O., \textit{Classification of quasi-minimal surfaces
with parallel mean curvature vector in pseudo-Euclidean 4-space
$\E^4_2$}, Result. Math., 2009, 55(1-2), 23--38.

\bibitem{Chen-Mihai}
Chen, B.-Y.,  Mihai, I., \textit{Classification of quasi-minimal slant
surfaces in Lorentzian complex space forms}, Acta Math. Hungar., 2009, 122(4), 307--328.


\bibitem{Chen-Yang}
Chen B.-Y., Yang D.,\textit{ Addendum to ''Classification of marginally
trapped Lorentzian flat surfaces in $\E^4_2$ and its application
to biharmonic surfaces''}, J. Math. Anal.  Appl., 2010, 361(1), 280--282.



\bibitem{GM2}
Ganchev G., Milousheva V., \emph{Invariants and Bonnet-type Theorem
for  Surfaces in $\mathbb{R}^4$}, Cent. Eur. J. Math. \textbf{8}  (2010)  993--1008.

\bibitem{GM6}
 Ganchev G.,  Milousheva V.,  \emph{An Invariant Theory of Marginally
Trapped Surfaces in the Four-dimensional Minkowski Space}, J. Math. Phys. \textbf{53}  (2012) Article ID: 033705, 15 pp.

\bibitem{GM-BKMS}
Ganchev G., Milousheva V., \emph{Special Classes of Meridian Surfaces in the
Four-dimensional Euclidean Space}, Bull. Korean Math. Soc. \textbf{52}, no. 6 (2015), 2035--2045.

\bibitem{GM-MC}
Ganchev G., Milousheva V., \emph{Meridian Surfaces of Elliptic or Hyperbolic Type in the
Four-dimensional Minkowski Space}, Math. Commun., \textbf{21}, no. 1 (2016),  1--21.



\bibitem{Liu-Liu-1}
Liu H., Liu G., \emph{Hyperbolic rotation surfaces of constant mean
curvature in 3-de Sitter space}, Bull. Belg. Math. Soc. Simon Stevin,  \textbf{7}, no. 3 (2000), 455--466.

\bibitem{Lop-2}
L\'{o}pez R., \emph{Timelike surfaces with constant mean curvature in Lorentz three-space}. Tohoku Math. J. \textbf{52} (2000), 515--532.

\bibitem{O'N}
O'Neill M., \textit{Semi-Riemannian geometry with applications to relativity},
Academic Press, London 1983.


\bibitem{Rosca}
Rosca, R., \textit{On null hypersurfaces of a Lorentzian manifold}. Tensor (N.S.) \textbf{23}  (1972), 66--74.

\bibitem{Sa}
Sasahara N., \emph{Spacelike helicoidal surfaces with constant mean curvature in Minkowski 3-space}. Tokyo J. Math. \textbf{23} (2000), no. 2, 477--502.


\end{thebibliography}
\end{document}